\documentclass[11pt]{amsart}
\usepackage{amsfonts}
\usepackage{amsmath}
\usepackage{amssymb,latexsym}
\usepackage[mathcal]{eucal}
\usepackage{amscd}

\input xy
\xyoption{all}

\usepackage[pdftex,bookmarks,colorlinks,breaklinks]{hyperref}

\newcommand{\Ucal}{\mathcal{U}}

\newcommand{\ch}{\mathbf{1}}

\newcommand{\N}{\mathbb{N}}

\newcommand{\Xb}{\mathbf{X}}

\newcommand{\al}{\alpha}
\newcommand{\Ga}{\Gamma}

\newcommand{\del}{\delta}
\newcommand{\Del}{\Delta}

\newcommand{\sig}{\sigma}

\newcommand{\om}{\omega}
\newcommand{\Om}{\Omega}

\newcommand{\ol}{\overline}

\newcommand{\br}{\vspace{3 mm}}
\newcommand{\imp}{\Rightarrow}
\newcommand{\nor}{\vartriangleleft}

\newcommand{\rest}{\upharpoonright}

\newcommand{\cls}{{\rm{cls\,}}}

\newcommand{\Homeo}{\rm{Homeo\,}}

\newcommand{\Sub}{\mathbf{Sub}}

\newcommand{\bcs}{\backslash}

\newcommand{\roott}{{\rm{root}}}

\swapnumbers
\theoremstyle{plain}
\newtheorem{thm}{Theorem}[section]

\newtheorem{prop}[thm]{Proposition}
\newtheorem{Claim}[thm]{Claim}

\theoremstyle{definition}
\newtheorem{defn}[thm]{Definition}

\newtheorem{prob}[thm]{Problem}

\begin{document}

\title[On minimal actions of countable groups]
{On minimal actions of countable groups}

\author{Eli Glasner and Benjamin Weiss}

\address{Department of Mathematics\\
     Tel Aviv University\\
         Tel Aviv\\
         Israel}
\email{glasner@math.tau.ac.il}
\address {Institute of Mathematics\\
 Hebrew University of Jerusalem\\
Jerusalem\\
 Israel}
\email{weiss@math.huji.ac.il}

\begin{date}
{December 26, 2017}
\end{date}

\begin{abstract}
Our purpose here is to review some recent developments 
in the theory of dynamical systems
whose common theme is a
link between minimal dynamical systems, certain
Ramsey type combinatorial properties, and the Lov\'{a}sz local lemma (LLL).
For a general countable group $G$ the two classes of minimal systems we will
deal with are (I) the minimal subsystems of the {\em subgroup system} $ (\Sub(G), G)$, 
called URS's (uniformly recurrent subgroups),
and (II) minimal {\em subshifts}; i.e. subsystems of the binary Bernoulli $G$-shift  
$(\{0, 1\}^G, \{\sig_g\}_{g \in G})$.
\end{abstract}



\maketitle





%
%




\section*{Introduction}

The following quotes are from the
1964 paper of Walter H. Gottschalk , one of the founders of the modern theory of topological dynamics, entitled {\em A survey of minimal sets}, \cite{Gott-64}.
\begin{quote}
As for any kind of mathematical structure, there are two basic problems regarding minimal sets:
(1) The classification problem; (2) The construction problem.
\end{quote}

 ...

\begin{quote}
Apart from a relatively few particular constructions of certain minimal sets which do not seem to generalize readily, there
appears to be presently two rather large classes of minimal sets :
(I) coset transformation groups; (II) symbolic flows.
\end{quote}

Of course in the year 1964 the theory of minimal sets was in a rather early stage of its development
and nowadays we have a vast literature on that subject, to which no reasonable survey can do justice.
Our purpose here is much less ambitious.
We would like to review some recent developments whose common theme is a
link between the general theory of minimal dynamical systems, certain
Ramsey type combinatorial properties, and the Lov\'{a}sz local lemma (LLL).

Notwithstanding the huge increase in the stock of examples of minimal systems we have today,
Gottschalk's observation about the two basic classes of examples (I) and (II) above remains,
to some extent, valid. In view of the recent developments we were referring to,
we will replace the class of cosets transformation groups (or homogenous dynamical systems,
in a more modern terminology) with the larger class of {\em subgroup systems}; i.e.
dynamical systems which arise as subsystems of the system $(\mathcal{S}(G), G)$, where
for a locally compact group $G$, the space $\mathcal{S}(G) = \Sub(G)$ is the compact space of closed
subgroups of $G$ equipped with the Fell
topology, and the action of $G$ on $\Sub(G)$
is by conjugation. As to the class of subshifts, there is no doubt that even today
 this is one of the most prominent and well studied classes of dynamical systems.
Thus the two classes of minimal systems we will
deal with are (I) the minimal subsystems of $(\mathcal{S}(G), G)$, called URS's (uniformly recurrent subgroups),
and (II) minimal symbolic dynamical systems; in fact for our purposes minimal subshifts
of the binary Bernoulli $G$-shift $(\Om, G)$ will suffice, where here $G$ is a discrete countable group,
$\Om = \{0,1\}^G$ and $G$ acts by permutations: $g\om(h) = \om(g^{-1}h), \ \om \in \Om,\  g, h \in G$.

In \cite{GW}, where the notion of URS was introduced, we showed that every URS is
the stability system of a transitive dynamical system and asked whether this transitive system can be chosen to be
minimal (for the exact definitions of these notions see Section \ref{Sec-stab} below).
Not long after, this question was addressed in three independent works, \cite{K}, \cite{Elek} and \cite{MB-T}.
In \cite{MB-T} there is a complete answer to the question which in fact provides much more.
The two other works give partial answers and use completely different methods.
In the first part of our review we chose to expound Elek's proof of his theorem which asserts
that the answer is positive in the case of a finitely generated group.
Elek's main tool is the LLL.

In \cite{GU} Glasner and Uspenskij asked whether every infinite countable group $G$
admits a free subshift $X$.
By a {\em subshift} we mean a closed subset $X$ of $\Om = \{0, 1\}^G$ that
is invariant under the shift action. It is said to be {\em free} if for all $e \not = g \in G$,
and for all $x \in X$, $gx \not= x$.
In that paper the question was also formulated as
a coloring problem.
At about the same time Gao, Jackson and Seward in \cite{GJS}, provided a positive answer to this question
by means of an intricate bare hands construction.
Recently, again using the LLL as a main tool, Aubrun, Barbieri and Thomass\'{e} \cite{ABT},
gave an alternative, much shorter proof, a variation of which we will present
in the second part of our review.
%
We end this introduction by recalling the LLL.

\begin{thm} [Lov\'{a}sz's Local Lemma]\label{LLL}
Let $X$ be a finite set and $P$ be a probability distribution on $X$.
Let $\mathcal{A} = \bigcup_{i=1}^r \mathcal{A}_i$ be a  collection of subsets of $X$.
Suppose that for all $A \in  \mathcal{A}_i, \ P(A) = p_i$ and that
there are real numbers $0 \le a_1, a_2, \dots, a_r < 1$ and $\Del_{ij} \ge 0, \ i,j = 1,2, \dots, r$,
such that the following conditions hold:

\begin{enumerate}
\item
for any $A \in \mathcal{A}_i$ there exists a set
$\mathcal{D}_A \subset \mathcal{A}$  with
$|\mathcal{D}_A  \cap \mathcal{A}_j| \le \Del_{ij}$ for all $1 \le j \le r$
such that $A$ is independent of $\mathcal{A} \setminus (\mathcal{D}_A \cup {A})$,
\item
$p_i \le a_i  \prod_{j =1}^r (1 - a_j)^{\Del_{ij}} \
{\text{ for all}} \ 1 \le i \le  r$.
\end{enumerate}
Then $P(\bigcap _{A \in \mathcal{A}} A^c) > 0$.
\end{thm}


\br


\section{Stability systems and URS}\label{Sec-stab}

Let $G$ be a locally compact second countable topological group.
A {\em $G$-dynamical system} is a pair $(X,G)$ where $X$ is a compact metric
space and $G$ acts on $X$ via a continuous homomorphism
$G \to \Homeo(X)$ the Polish groups of self homeomorphisms of $X$
equipped with the topology of  uniform convergence.
Given a compact dynamical system $(X,G)$, for
$x \in X$ let
$$
St_x= G_x =\{g \in G : gx =x\}
$$
be the {\em stability group at $x$}.

Let $\mathcal{S}=\mathcal{S}(G)$ be the compact metrizable space of all subgroups of $G$
equipped with the Fell topology.
Recall that given a Hausdorff topological space $X$, a basis for the {\it Fell topology} on the hyperspace $2^X$, comprising the closed subsets of $X$, is given by the collection of sets $\{\Ucal(U_1,\dots,U_n; C)\}$, where
$$
\Ucal(U_1,\dots,U_n; C) =\{A \in 2^X :  \forall \ 1 \le j \le n,\  A \cap U_j
\not=\emptyset \ \& \ A \cap C = \emptyset\}.
$$
Here $\{U_1,\dots, U_n\}$ ranges over finite collections of open subsets of
$X$ and $C$ runs over the compact subsets of $X$.
The Fell topology is always compact and it is Hausdorff iff $X$ is locally compact (see e.g. \cite{Be}).
We let $G$ act on $\mathcal{S}(G)$ by conjugation,
 $(g,H) \mapsto g\cdot H = H^g = gHg^{-1}$
($g \in G, H \in \mathcal{S}(G)$).
This action makes $(\mathcal{S}(G),G)$ a $G$-dynamical system.

Perhaps the first systematic study of the space $\mathcal{S}(G)$ is to be found in
Auslander and Moore's memoir \cite{AM}. It then played a central role in the seminal work of
Stuck and Zimmer \cite{SZ}.
More recently the notion of IRS (invariant random subgroup)
was introduced in the work of M. Abert, Y. Glasner and B. Virag \cite{AGV}.
Formally this object is just a $G$-invariant probability measure on $\mathcal{S}(G)$.
This latter work together with A. Vershik \cite{V},
served as a catalyst and lead to a renewed vigorous interest in the study of IRS's
(see, among others, \cite{A-S}, \cite{Aetal}, \cite{AGV2},  \cite{B}, \cite{B2}, and \cite{BGK1}, \cite{BGK2}).
A brief historical discussion of the subject can be found in \cite{AGV}.

Pursuing the well studied analogies between ergodic theory and topological dynamics
(see \cite{GW}) we  introduced in \cite{GW-15} a topological dynamical analogue of the notion of an IRS.

\begin{defn}
A  {\em uniformly recurrent subgroup} (URS for short) is a minimal subsystem of $\mathcal{S}(G)$.
\end{defn}

In turn, \cite{GW-15} attracted the attention of several authors and interesting links with
the theory theory of simple $C^*$-algebras were established;
see \cite{Ke-15}, \cite{LB-MB}, \cite{K}, \cite{MB-T}, \cite{Elek}.

\br


It is easy to check that, whenever $(X,G)$ is a dynamical system,
the map $\phi : X \to \mathcal{S}(G),\ x \mapsto G_x$ is upper-semi-continuous.
In fact, if $x_i \to x$ and $g_i \to g$ (when $G$ is discrete the latter just means that
eventually $g_i=g$) are convergent sequences in $X$ and
$G$ respectively, with $g_i \in G_{x_i}$, then $x_i=g_i x_i \to gx$, hence
$gx=x$ so that $\limsup G_{x_i} \subset G_x$.

\begin{defn}
Let $\pi : (X,G) \to (Y,G)$ be a homomorphism of $G$-systems; i.e.
$\pi$ is a continuous, surjective map and $\pi(gx) = g \pi(x)$ for every $x \in X$ and $g \in G$.
We say that $\pi$ is an {\em almost one-to-one extension} if there is a dense $G_\del$
subset $X_0 \subset X$ such that $\pi^{-1}(\pi(x)) = \{x\}$ for every $x \in X_0$.
\end{defn}

\begin{prop}\label{stability}
Let
$(X,G)$ be a compact system.
Denote by $\phi : X \to \mathcal{S}(G)$ the upper-semi-continuous map $x \mapsto G_x$
and let $X_0 \subset X$ denote the dense $G_\del$ subset of continuity points of $\phi$.
Construct the diagram
\begin{equation*}
\xymatrix
{
& \tilde{X} = X \vee Z \ar[dl]_\eta\ar[dr]^\al  & \\
X & & Z
}
\end{equation*}
where
\begin{gather*}
Z = {\cls} \{G_{x} : x \in X_0\}\subset \mathcal{S}(G),\\
\tilde{X} = {\cls}\{(x,G_{x}) : x \in X_0\} \subset X \times Z,
\end{gather*}
and $\eta$ and $\al$ are the restrictions to $\tilde{X}$ of the projection maps.
We have:
\begin{enumerate}
\item
The map $\eta$ is an almost one-to-one extension.
\end{enumerate}
If moreover $(X,G)$ is minimal then
\begin{enumerate}
\item[(2)] $Z$ and $\tilde{X}$ are minimal systems.
\item[(3)]
$Z$ is the unique minimal subset of the set ${\cls} \{G_x : x \in X\}\subset \mathcal{S}(G)$
and $\tilde{X}$ is the unique minimal subset of the set
${\cls} \{(x,G_x) : x \in X\}\subset X \times \mathcal{S}(G)$.
\end{enumerate}
\end{prop}

\begin{proof}
(1)\
Let $\tilde{X}_0 =\{(x,G_x) : x \in X_0\}$. It is easy to see that the fact that $x \in X_0$
implies that the fiber $\eta^{-1}(\eta((x,G_x)))$ is the singleton $\{(x,G_x)\}$.

(2)\
Fix a point $x_0 \in X_0$. The minimality of $(X,G)$ implies that the orbit of the point
$(x_0,G_{x_0})$ is dense in $\tilde{X}$. On the other hand,
if $(x,L)$ is an arbitrary point in $\tilde{X}$ then, again by minimality of $(X,G)$, there
is a sequence $g_n \in G$ with $\lim_{n}g_nx = x_0$. We can assume that
the limit $\lim_n g_n (x,L) = (x_0,K)$ exists as well, and then the fact that $x_0 \in X_0$
implies that $K = G_{x_0}$.
This shows that $(\tilde{X},G)$ is minimal and then so is $Z = \al(\tilde{X})$.

(3)\
Given any $x\in X$ we argue, as in part (2), that $(x_0,G_{x_0})$ is in the orbit closure
of $(x,G_x)$.
\end{proof}

\begin{defn}\label{stab-sys-def}
Given a dynamical system $(X,G)$ we call the system $Z \subset \mathcal{S}(G)$
{\em the stability system of $(X,G)$}. We denote it by $Z = \mathcal{S}_X$.
We say that $(X,G)$ is {\em essentially free} if $Z = \{e\}$.
When $X$ is minimal $Z$ is a URS.
\end{defn}

A basic result proved in \cite{AGV} is that to every ergodic IRS $\nu$ (a $G$-invariant
probability measure on $\mathcal{S}(G)$) there corresponds an
ergodic probability measure preserving system $\Xb = (X,\mathcal{B},\mu,G)$ whose
stability system is $(\mathcal{S}(G),\nu)$.
In \cite{GW-15} we have shown that every URS is obtained as the stability system of a compact metric
topologically transitive system $(X, G)$ and posed the question
whether a stronger topological analogue
of the IRS result holds; namely:

\begin{prob}\label{prob}
Given a URS $Z \subset \mathcal{S}(G)$ is there always a metric minimal dynamical system
$(X, G)$ whose stability system is $Z$ ?
\end{prob}

\br

\section{Non-repetitive colorings}
Let $\Ga = (V,E)$ be a graph.
Call a coloring $c : V \to K = \{1,2, \dots , C\}$ of the set of vertices of a graph {\em nonrepetitive}
if for any simple path
$(x_1,x_2,...,x_{2n})$
in $\Ga$ (i.e. $\forall  i \not = j, \ x_i \not = x_j$)
there exists some $1 \le i \le n$ such that $c(x_i)\not= c(x_{n+i})$.

\br

The following theorem is essentially due to
Alon,  Grytczuk, Haluszczak and Riordan, \cite{AGHR}.

\begin{thm}\label{rep}
For any $d \ge 1$ there exists a constant $C(d) > 0$ such that any graph $\Ga$ (finite or infinite)
with vertex degree bounded by $d$ has a nonrepetitive coloring with an alphabet $K$, provided that $|K| \ge C(d)$.
\end{thm}

\begin{proof}
The proof is an application of the LLL, Theorem \ref{LLL}, as follows.
%
%
%
Let $\Ga$ be a finite graph with maximum degree $d$.
It is enough to prove our theorem for finite graphs.
Indeed, if $\Ga'$ is a connected infinite graph with vertex degree bound $d$,
then for each ball around a given vertex $v$ we have a nonrepetitive coloring.
Picking a pointwise convergent subsequence of the colorings we obtain
a nonrepetitive coloring of our infinite graph $\Ga'$.

Let $C$ be a large enough number, its exact value will be given later.
Let $X$ be the set of all random $\{1, 2, . . . , C\}$-colorings of $\Ga$.
Let $r = |V|$ and for
$1\le i \le r$ and for any path $L$ of length $2i -1$ let $A(L)$
be the event that $L$ is repetitive.
Set
$$
\mathcal{A}_i = \{A(L): L \ {\text{is a path of length $2i - 1$ in $\Ga$}}\},
$$
and let  $p_i = C^{- i}$.
The number of paths of length $2j - 1$ that intersects a given
path of length $2i - 1$ is less or equal than $4ijd^{2j}$ .
We set
$$
\Del_{ij} = 4ijd^{2j},\quad  \quad  a = \frac{1}{(2d)^{2}},    \quad  \quad      a_i =  \frac{1}{(2d)^{2i}}.
$$
Since $a_i \le 1/2$, we have $(1 - a_i) \ge \exp(- 2a_i)$.

In order to be able to apply the Local Lemma, we need that for any $1 \le i \le r$,
$$
p_i \le a_i \prod_{j=1}^r \exp(-2a_j \Del_{ij}).
$$
Now
$$
2a_j \Del_{ij} = 2\frac{1}{(2d)^{2j}} 4ij d^{2j} = 8 ij 2^{-2j}
$$
and the required inequality is
$$
C^{-i}  \le a_{i} \prod_{j=1}^r \exp(-8ij 2^{-2j}) =
[a \exp(-8\sum_{j=1}^r j2^{-{2j}})]^i,
$$
or equivalently
$$
C \ge a^{-1} \exp \left(8 \sum_{j=1}^r \frac{j}{2^{2j}} \right).
$$
Since the infinite series  $\sum_{j=1}^\infty \frac{j}{2^{2j}}$ converges,
we obtain that for large enough
$C$, the conditions of the Local Lemma are satisfied independently of the size of
our finite graph $\Ga$.
This ends the proof of Theorem
\ref{rep}.
\end{proof}

\br

\section{Realizations of URS's as stability systems}\label{Sec6}



The following theorem of Elek provides a partial answer to our question \ref{prob}.
We will reproduce his proof with some slight modifications.

\begin{thm}\label{real}
Let $G$ be a finitely generated infinite group.
For every URS $Z \subset \mathcal{S}(G)$ there is a minimal system $(X,G)$ with
$Z$ as its stability system.
\end{thm}

\begin{proof}
We begin with the case where some $H \in Z$ has finite index in $G$.
It then follows that $Z$ is finite and we let $Y = G/H = \{gH : g\in G\}$, the finite homogeneous
$G$-space of right $H$-cosets.
We now take $X$ to be the $G$-orbit of the point $x_0=(H,H)$ in the product system
$(Z \times Y, G)$. Note that $G$ acts on $Z$ by conjugation, whereas it acts on $Y$
by multiplication on the left. The stability subgroup $G_{x_0}$ is $H$ and it follows that
indeed $Z$ is the stability system of the (minimal) finite system $(X,G)$.

We now assume that $[G:H] = \infty$ for every $H \in Z$.

Let $K =\{1,2, \dots, C\}$ be a finite alphabet and set $\Om = K^G$.
Let $G$ act on $\Om$ by $g\om(h) = \om(g^{-1}h)$.
For $H < G$ let
$$
\Om_H = \{\om \in \Om : h\om = \om,\ \forall h \in H\},
$$
and
$$
X = \bigcup \{\{H\} \times \Om_H : H \in Z\} = \{(H,\om) \in Z \times \Om : h\om = \om, \ \forall h \in H\}.
$$
It is easy to check that $X$ is a closed $G$-invariant subset of
$Z \times \Om$.
Also note that the projection map $\pi : (X,G) \to (Z,G)$ is clearly a homomorphism of $G$-systems.

\br

1. \
Our first observation is that for each $x =(H,\om) \in X$ we have $\pi(x)=H$ and
$$
H \subset G_x = \{g \in G : gx=x\} \subset G_{\pi(x)}= G_H = \{g\in G : gHg^{-1} =H\} = N_G(H),
$$
the normalizer of $H$ in $G$.
Also note that
$H \nor N_G(H)$, so that $\tilde{H} =N_G(H)/H$ is a group.
Finally observe that for every $x \in X$ the group $\tilde{H}$, with $H =\pi(x)$, acts on the fiber
$\pi^{-1}(H) = \{H\} \times \Om_H \cong \Om_H$.

%
%
\br

2.\
Let $S = \{s_1, \dots, s_d\}$ be a fixed symmetric set of generators for $G$.
For any $H$ in $Z$ we will consider each element $x = (H,\om) \in \{H\} \times \Om_H \subset X$
as a coloring (with $C$ colors) of the corresponding Schreier graph $(V, E)$ whose set of vertices $V$ is the
collection of left $H$-cosets, $V = H\bcs G$, and a pair $(Ha, Hb)$ is an edge (an element of $E$)
iff $Hb = Has$ for some $s \in S$. We say that $x = (H, \om)$ is the {\em rooted colored Schreier graph
with root $H$} and write $H = \roott(x)$.

Note that for $g \in G$ and $x = (H,\om) \in X$
$$
g(H,\om) = (gHg^{-1}, g\om) = (gHg^{-1}, \om(g^{-1} \cdot))
$$
and thus $g\om$ is defined over the Schreier graph on $gHg^{-1} \bcs G = H^g\bcs G$,
where $H^ga$ and $H^gb$ are connected by an edge iff
$Ha$ and $Hb$ are connected by an edge in the Schreier graph  on $H \bcs G$.

Also note that if $g \in N_G(H)$ and $x=(H,\om) \in X$ then
$$
gx = (gH g^{-1}, g\om) = (H, g\om)
\quad {\text{with}} \quad
(g\om)(Ha) = \om(g^{-1}Ha) = \om(Hg^{-1}a).
$$
Finally note that for $g \in N_G(H)$, unless $g \in H$, we have for all $a \in G$,
$Hag \not = Ha$
so that for such $g$ the induced permutation, $\tilde{g}$ on $H\bcs G$,
has no fixed points ($Hga = Ha \imp Hg = H \imp g \in H$).
Moreover, for $g \in N_G(H) \setminus H$, if $g(H, \om) = (H, \om)$ then
the permutation $\theta = \theta_g$,
$$
 \theta_g(Ha) = g^{-1}Ha = Hg^{-1}a
$$
defined on $H \bcs G$, is an automorphism of the rooted colored graph $(H,\om)$ which is not the identity automorphism.

\br

3.\
Let $|K| = C(d)$ and fix some $H_0 \in Z$.  Let $\om_0 : H_0 \bcs G \to K$
be a non-repetitive $K$-coloring that gives
rise to an element $y_0 =(H_0,\om_0) \in  \{H_0\} \times \Om_{H_0}$.
Let $Y = \ol {\{gy_0 : g \in G\}}$ denote the orbit closure of $y_0$ in $X$.
The following proposition finishes the proof of
Theorem 1.

\br

\begin{prop}\label{pi=st}
For every $x = (H, \om) \in Y$ we have $St_x = H$; in other words
the the map $\pi(x) = St_x$.
Thus if $Y_0 \subset Y$ is any minimal subset of $Y$ then $St : Y_0 \to Z$
is also a surjective map, so that $Z$ is the stability system of the minimal system $Y_0$.
\end{prop}

\begin{proof}
Let $x = (H, \om) \in Y$ with underlying Schreier graph $H$ and coloring $\om : H  \bcs G \to K$.
Suppose that for some $g \in G$ we have $gx = x$; i.e. $g \in St_x$.
Then, in particular $H^g = H$, i.e. $g \in N_G(H)$
and if $g \not\in H$ then, as in step 2, $g$ defines an automorphism $\theta = \theta_g$
of the rooted colored Schreier graph $x$
such that
$\theta(H)  = Hg^{-1} \not= H$.

For convenience, for an element $a \in G$,
we will write $a \in V(x) = H \bcs G$, identifying $a$ with $Ha$.
We let $dist_x (a,b)$ denote the length of the shortest simple path from $a$ to $b$
in the graph $x$.
Let $a \in V(x)$ be a vertex such that
$dist_x (a,\theta(a))$ is minimal; this is $>0$ since $\theta$ has no fixed points.
Let $ (a_1,a_2,...,a_{n+1})$, with $a_1 = a$ and $a_{n+1} = \theta(a) = g^{-1}a$,
be a shortest path between $a$ and $\theta(a)$.
Let
$$
a_2 = a_1 s_{k_1},  \quad  a_3 = a_2 s_{k_2} ,\quad \dots \quad a_n= a_{n - 1}s_{k_{n-1}}
$$
Then let
$$
a_{n+2} = a_{n+1}s_{k_1}, \quad  a_{n+3} = a_{n+2}s_{k_2} ,\quad   \dots \quad a_{2n} = a_{2n - 1}s_{k_{n-1}}.
$$
Thus, for $1 \le i \le n$ we have
$a_{n+i} = g^{-1}as_{k_1}\cdots s_{k_i} = \theta(a_i)$,
so that
$$
(a_{n+1}, a_{n+2},  \dots,  a_{2n} )= (\theta (a_1), \theta(a_2),...,\theta(a_{n})),
$$
Since $g\om = \om$,
\begin{equation}\label{c=}
\om(a_i) = \om(a_{n + i}).
\end{equation}


\br

\begin{Claim}
The walk $(a_1,a_2,...,a_{2n})$ is a simple path.
\end{Claim}

\begin{proof}
Suppose that this walk crosses itself. Since $(a_1,a_2,...,a_{n+1})$ is the shortest pass
between $a$ and $\theta(a)$, for some $i, j, \ a_j =a_{n+i}$.
Now if $i = j$ this would imply that $a = \theta(a)$. On the other hand if either $i < j$ or $j < i$
then this would provide a shorter path between $a$ and $\theta(a)$.
Thus, in any case we arrive at a contradiction.
%
%
Therefore, $(a_1, a_2, . . . , a_{2n})$ is a path.
\end{proof}

By (\ref{c=}) and the previous Claim, the $K$-colored Schreier graph $x$ contains a repetitive path.
Since $x$ is in the orbit closure of $y$, this implies that $y$ contains a repetitive path as well,
in contradiction with our assumption. Thus we have shown that $gx =x$ implies $g \in H$.
This completes the proof of the proposition.
\end{proof}

The proof of Theorem \ref{real} is now complete.
\end{proof}

\br

\section{The work of Matte Bon and Tsankov}

In \cite{MB-T} Matte Bon and Tsankov provide a complete answer to Problem \ref{prob}.
In fact they show much more:

\begin{thm}
Let $G$ be a locally compact group and let
$H \subset \mathcal{S}(G)$ be a closed, invariant subset.
Then there exists a continuous action of $G$ on a compact space $X$ such
that the stabilizer map $St : X \to \mathcal{S}(G)$ is everywhere continuous and its image is
equal to $H$. If $G$ is second countable, $X$ can be chosen to be metrizable.
\end{thm}

Their proof is in the tradition the abstract theory of topological dynamics and follows
the ideas of the theorems of Ellis and Veech on the freeness, for a locally compact $G$,
of the universal minimal system $(M(G), G)$.
They also show that for every URS $Z \subset \mathcal{S}(G)$
there exists a unique universal minimal (usually non-metrizable)
system $(M_Z(G), G)$ for minimal $G$-systems subordinate to $Z$, in the following sense:

\begin{defn}
Let $G$ be a locally compact topological group.
\begin{enumerate}
\item
If $Y$ and $Z$ are URSs of $G$, we write $Y \prec Z$ if for every $H \in Y$, there exists
$K \in W$ such that $K  \le H$.
\item
A minimal system $(X,G)$ is {\em subordinate} to a urs $Z$ if
$Z \prec \mathcal{S}_X$.
\end{enumerate}
\end{defn}

\begin{thm}
For every URS $Z \subset \mathcal{S}(G)$, there exists a minimal
system $M_Z(G)$, unique up to isomorphism, which is subordinate to
$Z$ and is universal for minimal $G$-systems subordinate to $Z$.
Moreover, the stabilizer map $St : M_Z(G) \to Z$ is continuous.
\end{thm}


\br

\section{Free subshifts of the Bernoulli system $\Om = \{0,1\}^G$}
In \cite{GU} Glasner and Uspenskij asked whether every infinite countable group $G$
admits a free subshift $X \subset \Om$.  Some partial results
towards a positive solution were obtained and the following combinatorial characterization
(due to V. Pestov) was formulated.


\begin{thm}\label{char}
Let $G$ be an infinite countable group.
The following conditions are equivalent.
\begin{enumerate}
\item
$G$ acts freely on some subshift of $\{0,1\}^G$.
\item
There exists a $2$-coloring of $G$ with the following property.
For every $g \ne e$, there is a finite set $A = A(g)$ such that
for every $h \in G$ there is an $a \in A \cap g^{-1}A$ such that
$ha$ and $hga$ have different colors.
\end{enumerate}
\end{thm}

A bit later the same combinatorial characterization was observed by
Gao, Jackson and Seward in \cite{GJS}, where a positive answer to this question is
given by means of an intricate bare hands construction.
\begin{thm}
\label{char}
Every countable infinite group has the $2$-coloring property.
\end{thm}

Recently, using the LLL as a main tool, Aubrun, Barbieri and Thomass\'{e} \cite{ABT},
gave an alternative proof, a variation of which we will next present.
 (See also the work of Bernshteyn \cite{Ber}.)

\begin{proof}[Proof of the Gao-Jackson-Seward theorem]
Let $G =\{e=s_0, s_1, \dots, s_n,\dots\}$ be an inumeration of $G$.
Choose a sequence of subsets of $G$, $\{T_n\}_{n \in \N}$ with $|T_n| = Cn$, for some
natural number $C$ which will be determined later, and such that
$s_n T_n \cap T_n =\emptyset$ for all $n \in \N$.
Let $\Om = \{0,1\}^G$ and let $\mu =(\frac12 (\del_0 + \del_1))^G$ be the Bernoulli  measure on $\Om$.
For $g \in G$ we let
\begin{gather*}
A_{n,g} = \{\om \in \Om : \om \rest gT_n = \om \rest  gs_n  T_n\}, \\
\mathcal{A}_n = \{A_{n,g}\}_{g \in G}, \quad {\text{and}} \quad
\mathcal{A} = \bigcup_{n \in \N} \mathcal{A}_n.
\end{gather*}
Note that $\mu(A_{n,g})= 2^{-Cn}$.

Our purpose is to construct an element $\om_0 \in \Om$
such that for every $e \not= g \in G$ and every $\om \in \ol{\{g\om_0 : g \in G\}}$,
$g\om \not= \om$. Once we have such $\om_0$ any minimal subset
$X \subset \ol{\{g\om_0 : g \in G\}}$ will provide the required minimal free subshift.
To this end we will construct a point $\om_0 \in \Om$ such that for all $n \in \N$ and all $g \in G$,
$\om_0 \not \in A_{n,g}$.

The construction will proceed by repeated application of the LLL to the increasing sequence of finite subsets
$X_n = G_n = \{e=s_0, s_1, \dots, s_n\} \nearrow G$, and then by using a compactness argument.
Thus, it suffices to show that for a fixed $n$, the LLL inequality,
applied to $\mathcal{A} = \mathcal{A}_1 \cup \mathcal{A}_2 \cup\cdots \cup \mathcal{A}_n$, restricted to those $g \in G$
such that
$$
gT_k \cup gs_k T_k \subset X_n, \quad 1 \le k \le n
$$
is satisfied in $\{0,1\}^{X_n}$.

Thus we now have to check the LLL inequalies
$$
p_k \le a_k \prod_{l=1}^n (1- a_l)^{\Del_{k l}},   \quad 1 \le k \le n,
$$
on $\mathcal{A} = \bigcup_{k =1}^n \mathcal{A}_k$,
as follows:
\begin{enumerate}
\item[(i)]   we have $p_k = \mu(A) = 2^{-Ck}$ for all $A \in \mathcal{A}_k$,
\item[(ii)]  we choose $a_k =  2^{-\frac{Ck}{2}}$ ,
\item[(iii)]  and the numbers $\Del_{l k}$ are such that
 for every $k$ and every $A \in \mathcal{A}_k$, $A$ is independent of all $A' \in \mathcal{A}$ except for
$A'$ in a subset $\mathcal{D}_A \subset \mathcal{A}$ with
$|\mathcal{D}_A \cap \mathcal{A}_l | \le \Delta_{k,l}$, for $1 \le l \le n$.
\end{enumerate}

We next compute the numbers $\Delta_{k,l}$. So fix $g_0 \in G$
and for $A_{k, g_0}$ we need to compute
$\#\{A_{l,g}  \in \mathcal{A}  :  \  A_{l,g} \ {\text{ not independent of}} \ A_{k, g_0}\}$.
This condition means that the sets $gT_j \cup g s_l T_l$ and $g_0T_k \cup g_0 s_k T_k$ intersect
(e. g., if $t_1 \in T_l, \  t_2 \in T_k$ and
$gs_l t_1 = g_0 t_2$, then
$g = g_0 t_2t_1^{-1} s_l^{-1}$, etc.).
These sets have $2Cl$ and $2Ck$ elements respectively, whence we can take
$\Delta_{k,l} = 4C^2lk$.

Thus, the inequalities we have to check are:
\begin{equation}\label{in}
2^{-Ck} \le  2^{-\frac{Cn}{2}} \prod_{l=1}^n (1- 2^{-\frac{Cl}{2}})^{4C^2lk},
  \quad 1 \le k \le n.
\end{equation}

Taking into account the inequality $2^{-2x} \le 1 -x$, which holds for sufficiently small
nonnegative $x$, we get:
$$
1 - 2^{-\frac{Cl}{2}}  \ge 2^{-2 \cdot 2^{-\frac{Cl}{2}}},
$$
whence
$$
(1 - 2^{-\frac{Cl}{2}})^{4C^2lk}  \ge 2^{-8  C^2 lk 2^{-\frac{Cl}{2}}}.
$$
Thus in checking inequality (\ref{in}) it suffices to show that
$$
\prod_{l=1}^n (1- 2^{-\frac{Cl}{2}})^{4C^2lk} \ge
 2^{-8  C^2 k \sum_{l=1}^n 2^{-\frac{Cl}{2}}} \ge
 2^{-8  C^2 k \sum_{l=1}^{\infty} 2^{-\frac{Cl}{2}}}.
$$
As this is obviously the case for large enough $C$, our proof is complete.
\end{proof}

\br

\section{Invariant measures for minimal systems}

Recall that a countable group $G$ is {\em amenable} if whenever it acts continuously on a compact
space $X$ there is at least one probability measure $\mu$ on $X$ that is invariant under the action of $G$.
For nonamenable groups acting on $X$ there may or may not exist invariant probability measures.
In the memoir \cite{GJS2} where many kinds of subshifts are constructed the authors mention the
following problem that was raised by Ralf Spatzier:
\begin{prob}[11.2.6 in \cite{GJS2}]
Can the fundamental method be improved in order to construct a variety of subflows of $2^G$ which
support ergodic probability measures ?
\end{prob}

In particular one can ask does there exist any {\em free} subshift of $\{0, 1\}^G$ which supports an invariant
measure.
In the same paper \cite{Elek}, that we discussed earlier, G. Elek gives a positive answer in case the group $G$ is finitely generated
and sofic. In \cite{Ber} A. Bernshteyn proves more generally:

\begin{thm}\cite[Theorem 1.1]{Ber}
If $\Ga$ is a countably infinite group, then there exists a closed invariant subset
$X \subset \{0, 1\}^\Ga$
such that the action of $\Ga$ on $X$ is free and admits an invariant probability measure.
\end{thm}

He bases his proof of this result on a certain measurable version of the LLL that he establishes in \cite{Ber0}.
This says in essence that when we have a system of constraints on elements of $\{0, 1\}^G$ that guarantee
the freeness of the shift action as in the proof of the Gao-Jackson-Seward theorem, then for the
Bernoulli shift of $G$ over $[0, 1]^G$ with $\mu$ the product of Lebesgue measure, there is a measurable set
$A \subset [0, 1]^G$ with positive measure such that for $\mu$-a.e. $\om \in [0, 1]^G$ the name map $\pi$
from $[0, 1]^G \to \{0, 1\}^G$, given by
$\pi(\om)(g) = \ch_A(\sig_g\om)$, satisfies all of the constraints.
It then follows that the image of $\mu$ under $\pi$ gives the required invariant measure on the free subshift
which is the closed support of $\pi_*\mu$.
This method, as well as that of Elek, do not give minimal subshifts carrying an invariant measure.

In \cite{W}, for any countable group $G$, a
minimal system was constructed whose invariant measures represent any ergodic free measure preserving
action of $G$. The system is zero  dimensional  and thus, using any clopen subset, one gets a minimal subshift
$X$ of $\{0, 1\}^G$ with a host of invariant measures.
There is no guarantee that all points of this subshift are free, but there is certainly an
abundance of free points. It seems that the question as to the existence of a free minimal subshift
of $\{0, 1\}^G$ that supports an invariant probability measure for an arbitrary group is still open. 
\br

\end{document}